\documentclass[12pt]{article}
\usepackage{geometry}                
\usepackage{amssymb,amsmath, amsthm,latexsym, verbatim, amscd}
\usepackage{color}

\usepackage[parfill]{parskip}    
\usepackage{graphicx}
\usepackage{epstopdf}

\newtheorem{theorem}{Theorem}[section]
\newtheorem{prop}[theorem]{Proposition}
\newtheorem{lemma}[theorem]{Lemma}
\newtheorem{cor}[theorem]{Corollary}
\newtheorem{definition}[theorem]{Definition}
\newtheorem{convention}[theorem]{Convention}

\newcommand{\btheorem}{\begin{theorem}}
\newcommand{\etheorem}{\end{theorem}}
\newcommand{\bprop}{\begin{prop}}
\newcommand{\eprop}{\end{prop}}
\newcommand{\blemma}{\begin{lemma}}
\newcommand{\elemma}{\end{lemma}}
\newcommand{\bcor}{\begin{cor}}
\newcommand{\ecor}{\end{cor}}

\theoremstyle{remark}
\newtheorem{remark}[theorem]{Remark}

\newcommand{\bg}{{\mathfrak g }}

\newcommand{\bq}{{\mathfrak q }}

\newcommand{\bh}{{\mathfrak h }}

\newcommand{\bR}{{\mathbb R}}
\newcommand{\bN}{{\mathbb N}}
\newcommand{\bC}{{\mathbb C}}

\title{A hidden symmetry of a branching law}
\date{}                                           

\author
{
Toshiyuki Kobayashi
\vspace{-10mm}
\footnote
{
Graduate School of Mathematical Sciences,
 The University of Tokyo
 and Kavli IPMU (WPI),
 3-8-1 Komaba, Tokyo, 153-8914 Japan
\newline
\textit{E-mail address}:
\texttt{toshi@ms.u-tokyo.ac.jp}
}
\,\,
 and 
 Birgit Speh
\vspace{-10mm}
\footnote
{Department of Mathematics, Cornell University, Ithaca, NY 14853-4201, USA
\newline
\textit{E-mail address}:
\texttt{bes12@cornell.edu}}
}

\begin{document}

\maketitle

\begin{abstract}
We consider branching laws for the restriction of some irreducible unitary representations $\Pi$ of
$G=O(p,q)$ to its subgroup $H=O(p-1,q)$. 
In Kobayashi (arXiv:1907.07994, \cite{K}), 
 the irreducible subrepresentations of $O(p-1,q)$
 in the restriction of the unitary $\Pi|_{O(p-1,q) }$ are determined. By considering the restriction of packets of irreducible representations we obtain another very simple branching law, which was conjectured
 in {\O}rsted--Speh (arXiv:1907.07544, \cite{OS}).
\end{abstract}

\

{Mathematic Subject Classification (2020):} 
Primary 22E46; 
Secondary 
22E30, 
22E45,
22E50

\section{\bf Introduction}

The restriction of a finite-dimensional irreducible representation $\Pi^G$
 of a connected compact Lie group $G$ to a connected Lie subgroup $H$ is a classical problem.  
For example, 
 the restriction of irreducible representations of $SO(n+1)$ to 
the subgroup $SO(n)$ 
can be expressed as a combinatorial pattern satisfied by the highest weights of
 the irreducible representation $\Pi^G$ of the large group 
and of the irreducible representations appearing
 in the restriction of $\pi^H$ \cite{Weyl97}. 
For the pair $(G,H)=(S O(n+1), S O(n))$, 
 the branching law is always multiplicity-free, 
 {\it{i.e.,}}
\[
\mbox{dim Hom}_H(\pi^H,\Pi^G|_H)\leq 1.
\]

\medskip
In this article we consider a family of infinite-dimensional irreducible representations $\Pi^{p,q}_{\delta,\lambda}$
 with parameters $\lambda \in {\mathbb{Z}} + \frac 1 2(p+q)$,
 and $\delta \in \{ +,-\}$ of noncompact orthogonal groups $G=O(p,q)$
  with $p \ge 3$ and $q \ge 2$, 
 which have the same infinitesimal character as a finite-dimensional representation
 and which are subrepresentations of $L^2(O(p,q)/O(p-1,q))$
 for $\delta=+$,
 respectively of $L^2(O(p,q)/O(p,q-1))$
 for $\delta=-$. 
We shall assume a {\it{regularity condition}} of the parameter $\lambda$
 (Definition \ref{def:admissible}).  
Similarly we consider a family of infinite-dimensional irreducible 
 unitary representations $\pi^{p-1,q}_{\varepsilon,\mu}$, 
 $\varepsilon \in \{+, -\}$
 of noncompact orthogonal groups $H=O(p-1,q)$.

Reviewing the results of \cite{K} we see in Section IV
 that the restriction of these representations to the subgroup $H=O(p-1,q)$
 is either of \lq\lq{finite type}\rq\rq\
 (Convention \ref{conv:finite})
 if $\delta= +$
 or of \lq\lq{discretely decomposable type}\rq\rq\
 (Convention \ref{conv:deco})
 if $\delta =-$. 
If the infinitesimal characters of $\Pi_{\delta,\lambda}^{p,q}$ and of a direct summand of $(\Pi_{\delta,\lambda}^{p.q})|_{H}$ satisfy an interlacing condition
 \eqref{eqn:inter2} similar to that of the finite-dimensional representations
 of $(SO(n+1),SO(n))$, 
 then $\delta=+$ and the restriction of a representations $\Pi_{\delta,\lambda}^{p,q}$ is of finite type.  
On the other hand, 
 if the infinitesimal characters $\Pi_{\delta,\lambda}^{p,q}$ and of a direct summand of $(\Pi_{\delta,\lambda}^{p.q})|_{H}$ satisfy another interlacing condition \eqref{eqn:inter} similar to those of the holomorphic discrete series  representations of $(SO(p,2),SO(p-1,2))$, 
 then $\delta=-$
 and the restriction of a representations $\Pi_{\delta,\lambda}^{p,q}$ is
 of discretely decomposable type.

For each $\lambda$ we define a packet $\{ \Pi^{p,q}_{+,\lambda},\Pi^{p.q}_{-,\lambda} \}$ of representations with the same infinitesimal character. 
For simplicity, 
 we assume $p \ge 3$ and $q \ge 2$.  
Using the branching laws for the individual representations we show in Section V:

\begin{theorem} 
Let $(G,H)=(O(p,q), O(p-1,q))$.  
Suppose that $\lambda$ and $\mu$ are regular parameters.
\begin{enumerate}
\item[{\rm{(1)}}]
Let  $\Pi_{\lambda}$ be a representation in the packet $\{ \Pi_{+,\lambda}, \Pi_{-,\lambda} \}$. There exists exactly one representations  $\pi_{\mu } $ in the packet $\{\pi_{+,\mu}, \pi_{-,\mu} \}$
 so that
\[\dim \operatorname{Hom}_{H}(\Pi_{\lambda}|_H,\pi_{\mu}) = 1.\]

\item[{\rm{(2)}}]
Let $\pi_{\mu }$ be in the packet $\{\pi_{+,\mu}, \pi_{-,\mu} \}$. 
There exists exactly one representation  $\Pi_{\lambda} $ in the packet $\{ \Pi_{+,\lambda}, \Pi_{-,\lambda} \}$ so that 
 \[\dim \operatorname{Hom}_{H}(\Pi_{\lambda}|_H,\pi_{\mu}) = 1.\]
\end{enumerate}
\end{theorem} 

Equivalently we may formulate the result as follows:
\begin{theorem} 
[Version 2]
Suppose that $\lambda $ and  $\mu $ are regular parameters. Then
\[
   \dim \operatorname{Hom}_H((\Pi_{+,\lambda} \oplus \Pi_{-, \lambda})|_H, (\pi_{+,\mu} \oplus \pi_{-,\mu}))= 1. 
\]

\end{theorem}

Another version of this theorem using interlacing properties of infinitesimal characters is stated in 
Section V.

\medskip

{\bf Acknowledgements} 
The authors would like to acknowledge support by the MFO during  research in pairs stay during which part of this work was accomplished.

The first author was partially supported by Grant-in-Aid for Scientific Research (A) (18H03669), Japan Society for the Promotion of Science.\\
The second author was partially supported by Simons Foundation collaboration grant, 633703.  

{\bf Notation:}: $\bN = \{0,1,2,\dots ,\} $ and  $\bN_+ = \{1,2, \dots ,\}$.  

\bigskip


\section{\bf Generalities}

We will use in this article the notation and conventions of \cite{K} which we recall now. These conventions  differ from those used in \cite{OS}.

\medskip
Consider the standard quadratic form on ${\mathbb{R}}^{p+q}$
\begin{equation}
\label{eqn:quad}
  Q(X,X)=    x_1^2 + \dots +x_p^2 -x_{p+1}^2 -\dots -x_{p+q}^2
\end{equation}
 of signature $(p,q)$ in a basis $e_1, \dots ,e_p, e_{p+1},\dots e_{p+q}$.
We define $G=O(p,q)$
 to be the indefinite special orthogonal group
 that preserves the quadratic form $Q$.  
Let $H$ be the stabilizer
 of the vector $e_{1
 }$.  
Then $H$ is isomorphic to $O(p-1,q)$.  

Consider another quadratic form on ${\mathbb{R}}^{p+q}$
\begin{equation}
\label{eqn:quad-}
  Q_-(X,X)=    x_1^2 + \dots +x_q^2 -x_{q+1}^2 -\dots -x_{p+q}^2
\end{equation}
 of signature $(q,p)$ with respect to a basis $e_{-,1}, \dots ,e_{-,p}, e_{-,p+1},\dots e_{-,p+q}$. The orthogonal group $G_- = O(q,p)$ that preserves the quadratic form $Q_-$ is conjugate to $O(p,q)$ in $GL(p+q,\bR)$. Thus we may consider representations of $G_-= O(q,p)$ as representations of $G=O(p,q).$

Since $G$ and $G_-$ are conjugate, 
 the subgroup $H$ of $G$ is also conjugate to a subgroup $H_- $ of $G_-$ 
 which is isomorphic to $ O(q,p-1)$.   
This group isomorphism induces an isomorphism 
 of homogeneous spaces $G/H =O(p,q)/O(p-1,q) $ and $G_-/H_-= O(q,p)/O(q,p-1)$.  
  On the other hand  $O(p,q)/ O(p-1,q)$ and $O(q,p) / O(q-1,p)$  are not even homeomorphic to each other
 if $p \ne q$.
In the rest of the article we will  assume that   the subgroup $H_-$  preserves  the vector $e_{-,p+q}$.

The maximal compact subgroups of $G$, $G_-$ and $H,$ $H_-$ are denoted by  $K, $ $K_-$  respectively $K_H$ , $K_{H_-}$. The Lie algebras of the groups are denoted by the corresponding lowercase Gothic letters.

{\em To avoid considering special cases we make in this article the following:

{\bf Assumption ${\mathcal O }$:}
  
\begin{center}
              $p \ge 3$ and $q \ge 2$.  
\end{center}
}

\section{ \bf Representations}

We consider in this article  a family  of irreducible unitary  representations introduced in \cite{K}. 
Using the notation in \cite{K} we recall  their parametrization and  some important properties in this section.
The main reference is \cite[Sect.~2]{K}.

The irreducible unitary subrepresentations of $L^2(O(p,q)/O(p-1,q))$   were considered by many authors after the pioneering work
 by I.~M.~Gelfand et. al. \cite{Gelfand5}, 
 T.~Shintani, V.~Molchanov,
 J.~Faraut \cite{F79}, 
 and R.~Strichartz  \cite{S83}.  
For $p \ge 2$ and $q \ge 1$, 
 they are parametrized by $\lambda \in {\mathbb{Z}}+\frac 1 2(p+q)$ 
 with $\lambda>0$.  
Following the notation of \cite{K}  we denote them  by 
\[\Pi^{p,q}_{+,\lambda}.\]
They have infinitesimal character
\[ 
   (\lambda, \frac{p+q}{2}-2, \frac{p+q}{2}-3, \dots, \frac{p+q}{2}-[\frac{p+q}{2}]), 
\] 
in the Harish-Chandra parametrization
 (see \eqref{eqn:HC} below),
 and the minimal $K$-type 
\begin{equation}
\label{eqn:minK}
\begin{cases}
{\mathcal{H}}^{b(\lambda)}({\mathbb{R}}^p) \boxtimes {\bf{1}}
\qquad
&\text{if $b(\lambda) \ge 0$, }
\\
{\bf{1}} \boxtimes {\bf{1}}
\qquad
&\text{if $b(\lambda) \le 0$, }
\end{cases}
\end{equation}
where $b(\lambda):=\lambda-\frac 1 2(p- q -2)$ $(\in {\mathbb{Z}})$
 and ${\mathcal{H}}^b({\mathbb{R}}^p)$ stands for the space
 of spherical harmonics
 of degree $b$.  
We note
 that $\Pi_{+, \lambda}^{p,q}$ are so called Flensted-Jensen representations discussed in \cite{FJ}
 if $b(\lambda) \ge 0$, 
 namely,
 if $\lambda \ge \frac 12 (p-q-2)$.  
This is the case if $\lambda$ is regular 
 (Definition \ref{def:admissible}).  
The underlying $(\bg,K)$-module
 of $\Pi_{+, \lambda}^{p,q}$
 is given by a Zuckerman derived functor module. 
See \cite[Thm.~3]{K92} or \cite[Sect.~2.2]{K}.  

\begin{remark}
\label{rem:pone}
When $p=1$ and $q \ge 1$, 
 there are {\it{no}} irreducible subrepresentations
 in $L^2(O(p,q)/O(p-1,q))$, 
 and we regard $\pi_{+,\lambda}^{p,q}$ 
 as zero in this case.  
\end{remark}

\begin{remark}
{\rm{(1)}}\enspace
For any $p \ge 2$, $q \ge 1$
 and ${\mathbb{Z}}+ \frac 12 (p+q) \ni \lambda>0$, 
 the representation $\Pi_{+, \lambda}^{p,q}$
 of $G=O(p,q)$ stays irreducible 
 when restricted to $SO(p,q)$, 
 see also Remark \ref{rem:Arthur}.  
\par\noindent
{\rm{(2)}}\enspace
If $p=2$ and $\lambda \ge \frac1 2(p+q-2)$,  
 then the representation $\Pi^{p,q}_{+,\lambda}$ is a direct sum
 of a holomorphic discrete series representation
 and an anti-holomorphic discrete series representation
 when restricted to the identity component 
 $G_0=SO_0(p,q)$ of $G$.  
\end{remark} 

\medskip
Similarly there exist a family of irreducible unitary subrepresentations 
\[
   \Pi^{q,p}_{+,\lambda}
   \qquad
  (\lambda \in {\mathbb{Z}}+\frac 1 2(p+q), \,\, \lambda>0)
\]
of $G_-=O(q,p)$ in $L^2(G_-/H_-)=L^2(O(q,p)/O(q-1,p))$
 when $p \ge 1$ and $q \ge 2$, 
 with the same infinitesimal character and the same properties. 
Via the isomorphism between $(G_-, H_-)$ and $(G,H)$, 
 we may consider them as representations of $G=O(p,q)$ and irreducible subrepresentations of $L^2(G/H)=L^2(O(p,q)/O(p,q-1))$.

If no confusion is possible we use the simplified notation
\begin{align*}
\Pi_{+,\lambda} &= \Pi^{p,q}_{+,\lambda}
\intertext{and} 
   \Pi_{-,\lambda}&\simeq \Pi^{q,p}_{+,\lambda}
\qquad
\text{(via $G_- \simeq G$), }
\end{align*}
to denote representations of $G=O(p,q)$.  

\begin{remark}
\label{rem:tempered}
The irreducible representation $\Pi_{+,\lambda}$ are nontempered
 if $p \ge 3$, 
 and $\Pi_{-,\lambda}$ are nontempered
 if $q \ge 3$.  
\end{remark}

\begin{lemma}
\label{lem:III.2}
Assume that $\lambda \geq  \frac{1}{2}(p+q-2)$.
The representations $\Pi_{+,\lambda}$,  $\Pi_{-,\lambda}$ are inequivalent, but have the same infinitesimal character.  
\end{lemma}
\begin{proof}
The representation $\Pi_{+,\lambda}$ and $\Pi_{-,\lambda}$ are irreducible representations of $G=O(p,q)$
 with respective minimal $K$-types
\begin{alignat*}{2}
{\mathcal{H}}^b({\mathbb{R}}^p) &\boxtimes {\bf{1}}, 
\qquad
&&
b:=\lambda-\frac 1 2(p-q-2), 
\\
{\bf{1}} &\boxtimes {\mathcal{H}}^{b'}({\mathbb{R}}^q), 
\qquad
&&
b':=\lambda-\frac 1 2(q-p-2), 
\end{alignat*}
because the assumption $\lambda \ge \frac 1 2(p+q-2)$ implies
 both $b \ge 0$ and $b'\ge 0$ by \eqref{eqn:minK}.  
\end{proof}

\begin{remark}
Lemma \ref{lem:III.2} holds in the more general setting 
 where $\lambda \ge 0$, 
 see \cite[Thm.~3 (4)]{K92} 
 for the proof.  
\end{remark}

\begin{remark}
\label{rem:Arthur} 
For $p$ and $q$ positive
 and even, 
 the restriction of the representations $\Pi_{+,\lambda}$,  $\Pi_{-,\lambda}$ to $SO(p,q)$ are in an Arthur packet as discussed in \cite{AJ, MR}.  
Global versions of Arthur packets were introduced by J.~Arthur
 in the theory of automorphic representations 
 and are inspired by the trace formula \cite{A1, A2}.  
Our considerations of Arthur packets of representations
 of the orthogonal groups which are discrete series representations
 for symmetric spaces are inspired by Arthur's considerations as well as by the conjectures of B.~Gross and D.~Prasad.  
In this article we will refer to $\{ \Pi_{+,\lambda},\Pi_{-,\lambda} \}$  as a {\bf packet} of irreducible representations.
\end{remark}

Similarly we have $\mu \in {\mathbb{Z}}+\frac 1 2(p+q-1)$
 satisfying $\mu \geq  \frac{1}{2}(p+q-3)$ a packet $\{ \pi_{+,\mu}, \pi_{-,\mu} \}$ of unitary irreducible representations 
 of $G'=O(p-1,q)$.

\medskip

\begin{definition}
\label{def:admissible}
We say $\lambda \in {\mathbb{Z}}+\frac 1 2(p+q)$ respectively
 $\mu \in {\mathbb{Z}}+\frac 1 2(p+q-1)$
 are {\bf regular}
 if $\lambda \ge \frac 1 2 (p+q-2)$
 respectively 
 $\mu \ge \frac 1 2 (p+q-3)$.  
\end{definition}

\medskip

\begin{remark}
The irreducible representation $\Pi_{+,\lambda}$
 (or $\Pi_{-,\lambda}$) has the same infinitesimal character
 as a finite-dimensional irreducible representation
 of $G=O(p,q)$
 if and only if $\lambda \ge \frac 1 2 (p+q-2)$, 
 namely, 
 $\lambda$ is regular.  
Similarly, 
 $\pi_{+,\mu}$ (or $\pi_{-,\mu}$)
 has the same infinitesimal character 
 with a finite-dimensional representation 
 of $G'=O(p-1,q)$ 
 if and only if $\mu \ge \frac 1 2 (p+q-3)$, 
 namely, $\mu$ is regular.  
\end{remark}

For later use we define for regular $\lambda $ and $\mu $ the reducible  representations
\begin{equation}
\label{eqn:U}
 U(\lambda) = \Pi_{+,\lambda} \oplus \Pi_{-,\lambda}
\end{equation}
and 
\begin{equation}
\label{eqn:V}
 V(\mu) = \pi_{+,\mu} \oplus \pi_{-,\mu}.  
\end{equation}
of $G=O(p,q)$ respective of $H= O(p-1,q)$.

\section{\bf Branching laws}
In this section we summarize  the results of \cite{K}.  
For simplicity,
 we suppose that the assumption $\mathcal O$ is satisfied, 
 namely,
 we assume $p \ge 3$ and $q \ge 2$.  
We note
 that the results in Section \ref{subsec:IV.2} hold in the same form
 for $p \ge 2$ and $q \ge 2$, 
 and those in Section \ref{subsec:IV.3} hold for $p \ge 3$ and $q \ge 1$.

\subsection{Quick introduction to branching laws}
Consider the restriction of a unitary representation $\Pi$ of $G$ to a subgroup $G'$. 
We say that an irreducible unitary representation $\pi$ of $H$ is in the discrete spectrum of the restriction $\Pi|_{H}$
 if there exists an isometric $H$-homomorphism $\pi \to \Pi|_H$, 
 or equivalently, 
 if 
\[
\mbox{Hom}_H(\pi, \Pi|_{H}) \not = \{0\} 
\]
 where $\operatorname{Hom}_H(\, , \,)$ denotes
 the space of continuous $H$-homomorphisms.  
We define the multiplicity for the unitary representations
 by 
\[
   m(\Pi,\pi):= \dim \operatorname{Hom}_H(\pi, \Pi|_{H})
              = \dim \operatorname{Hom}_H(\Pi|_{H}, \pi).  
\]
\begin{remark}
As in \cite{GP, KS4}, 
 we also may consider the multiplicity $m(\Pi^{\infty},\pi^{\infty})$
 for smooth admissible representations $\Pi^{\infty}$ of $G$
 and $\pi^{\infty}$ of $G'$
 by 
\[
   m(\Pi^{\infty},\pi^{\infty}):= \dim \operatorname{Hom}_H(\Pi^{\infty}, \pi^{\infty}).  
\]
In general, 
 one has
\[
m(\Pi^{\infty},\pi^{\infty})
\ge 
m(\Pi,\pi).  
\]
\end{remark}

Besides the discrete spectrum there may be also continuous spectrum. 
Here are two interesting  cases:
\begin{enumerate}
\item 
There is no continuous spectrum and the representation $\Pi$ is a direct sum of irreducible  representations of $H$, 
{\it{i.e.}}, 
the underlying Harish-Chandra module is a direct sum of countably many Harish-Chandra modules of ($\bh,K_H$). 
We say that the restriction $\Pi|_H$ is {\it{discretely decomposable}}. 

\item There is continuous spectrum and there are only finitely many representations
 in the discrete spectrum
 in the irreducible decomposition
 of the restriction $\Pi|_H$.  
\end{enumerate}

We refer to the necessary and sufficient conditions  of the parameters of the irreducible representations $\Pi,\pi$
 so that $m(\Pi,\pi) \ne 0$
 (or $m(\Pi^{\infty},\pi^{\infty}) \ne 0$)
 as a {\it{branching law}}. 
In the examples below, 
 $m(\Pi^{\infty},\pi^{\infty})$, 
 $m(\Pi,\pi)\in \{0,1\}$
 for all $\Pi$ and $\pi$.   

\noindent {\em Examples of branching laws:}
\begin{enumerate}
\item  
Finite-dimensional representations of semisimple Lie groups are parametrized by highest weights. 
The classical branching law of the restriction of finite-dimensional representations of $SO(n)$ to $SO(n-1)$ is phrased as an interlacing pattern of highest weights, 
 see Weyl \cite{Weyl97}.  

\item The Gross--Prasad conjectures for the restriction of discrete series representations of $SO(2m,2n)$ to $SO(2m -1,2n)$ are expressed as interlacing properties of their parameters, 
 see \cite{GP}.  

\item The branching laws for the restriction of irreducible self-dual representations $\Pi^\infty$ of $SO(n+1,1)$ to $SO(n,1)$
 are expressed by using {\it{signatures}},
 {\it{heights}} and interlacing properties of weights, 
 see \cite{KS4}.  

\end{enumerate}

If $\Pi \in\{\Pi_{+,\lambda}, \Pi_{-,\lambda} \}$, and
\[
   \mbox{Hom}_{H }(\pi_H, \Pi|_{H}) \not  = \{0\} 
\]
then 
for a character $\chi$ of O(1) 
\[
   \mbox{Hom}_{H \times O(1)}(\pi_H \boxtimes \chi, \Pi|_{H \times O(1)}) \not =\{0\}.
\]
Moreover, 
 by \cite[Thm.~1.1]{K}   there exists a regular $\mu$ 
 so that $\pi_{H} \in \{\pi_{+,\mu},  \pi_{-,\mu} \}$.

If $\Pi $ is in the packet $\{\Pi_{+,\lambda}, \Pi_{-,\lambda}\}$
 and $\pi $ in the packet $ \{\pi_{+,\mu}, \pi_{-,\mu}\}$ the branching laws discussed in the next part will involve  the parameters $\lambda, \mu, \varepsilon, \delta$.

\subsection{Branching laws for the restriction of $\Pi_{-,\lambda}$ to \\ $H= O(p-1, q)$ --- discretely decomposable type}
\label{subsec:IV.2}
This section treats the restriction $\Pi_{-,\lambda}|_H$, 
 which is discretely decomposable.  
We use the explicit branching law given in \cite[Example 1.2 (1)]{K}. 
The results were also obtained in \cite{K2}
 by using different techniques, 
 see \cite{KSunada, KHowe70} for details.

We begin with the pair $(G_-, H_-)=(O(q,p),O(q,p-1))$.  
The restriction of the representation $\Pi_{+,\lambda}^{q,p}$ of $G_-$ to the subgroup
 $H_- \times  O(1)= O(q,p-1) \times O(0,1)$
is a direct sum of irreducible representations, 
 and is isomorphic to the Hilbert direct sum of countably many Hilbert spaces:
\[
   \bigoplus_{n  \in \bN} \pi_{+,\lambda +n+ \frac{1}{2}}^{q,p-1} \boxtimes (\operatorname{sgn})^{n}
\]
where $\operatorname{sgn}$ stands
 for the nontrivial character of $O(1)=O(0,1)$.  
Then via the identification $(G_-, H_-) \simeq (G,H)=(O(p,q), O(p-1,q))$
 and $\Pi_{+,\lambda}^{q,p} \simeq \Pi_{-,\lambda}$
 as a representation of $G_- \simeq G$, 
 we see the restriction of $\Pi_{-,\lambda}$ to $H\times O(1) = O(p-1,q) \times O(1,0)$ is discretely decomposable, 
 and we have an isomorphism
\[
   \Pi_{-,\lambda}|_H \simeq 
   \bigoplus_{n  \in \bN}~ \pi_{-,\lambda +n+ \frac{1}{2}} \boxtimes (\operatorname{sgn})^{n}.
\]

Hence 

\begin{prop}
[Version 1]
\label{prop:disc}
The restriction of $\Pi_{-,\lambda}$  to $H = O(p-1,q)$ is 
a Hilbert direct sum
\[
   \bigoplus_{n \in \bN } ~ \pi_{-,\lambda+n+ \frac{1}{2}} 
\]
 and each representation has multiplicity one.
\end{prop}

\medskip
\begin{remark}
\label{rem:disc}
If $\lambda$ is regular, 
 then $\mu$ is regular whenever
 $\operatorname{Hom}_{H}(\pi_{-,\mu}, \Pi_{-, \lambda}|_H) \ne \{0\}$.  
In contrast, 
 an analogous statement fails 
 for the restriction $\Pi_{+, \lambda}|_H$, 
 see Remark \ref{rem:finadm} below.  
\end{remark}

\medskip
\begin{remark}
If $G = SO_0(p,2)$ the representation $\Pi_{-,\lambda }$ 
 with $\lambda$ regular is a holomorphic discrete series representation. 
In this case, 
 this result follows from the work of H.~Plesner-Jacobson and M.~Vergne \cite[Cor.~3.1]{JV}
 or as a special case of the general formula
 proved in \cite[Thm.~8.3]{mf-korea}.
\end{remark}
 
We define $\kappa \colon {\mathbb{N}} \to \{0, \frac 1 2\}$ by
\[
   \text{$\kappa(n)=0$\qquad for $n$ even;
   $=\frac 1 2$\qquad for $n$ odd.}
\]
Then the infinitesimal character of the representation $\Pi_{-,\lambda}$ of $G$ is 
\begin{equation}
\label{eqn:Ginf}
   (\lambda, \frac{p+q -4}{2}, \dots ,\kappa(p+q)), 
\end{equation}
and the infinitesimal character of the representations in 
$\pi_{-+,\mu}$ of $H$ is 
\begin{equation}
\label{eqn:Hinf}
(\mu, \frac{p+q-5}{2}, \dots ,\kappa(p+q-1)).  
\end{equation}
Here we note
 that the groups $G$ and $H$ are not of Harish-Chandra class, 
 but the infinitesimal characters
 of the centers ${\mathfrak{Z}}_G({\mathfrak{g}}):=U({\mathfrak{g}})^G$
 and ${\mathfrak{Z}}_H({\mathfrak{h}}):=U({\mathfrak{h}})^H$
 of the enveloping algebras
 can be still described 
 by elements of ${\mathbb{C}}^M$ with $M:=[\frac 1 2 (p+q)]$
 and ${\mathbb{C}}^N$ with $N:=[\frac 1 2 (p+q-1)]$
 modulo finite groups 
 via the Harish-Chandra isomorphisms:
\begin{align}
\label{eqn:HC}
  \operatorname{Hom}_{\text{${\mathbb{C}}$-alg}}
  ({\mathfrak{Z}}_G({\mathfrak{g}}), {\mathbb{C}})
  \simeq\,\,&
  {\mathbb{C}}^M/{\mathfrak{S}}_M \ltimes ({\mathbb{Z}}/2 {\mathbb{Z}})^M, 
\\
\notag
 \operatorname{Hom}_{\text{${\mathbb{C}}$-alg}}
  ({\mathfrak{Z}}_H({\mathfrak{h}}), {\mathbb{C}})
  \simeq\,\,&
  {\mathbb{C}}^N/{\mathfrak{S}}_N \ltimes ({\mathbb{Z}}/2 {\mathbb{Z}})^N.  
\end{align}
In our normalization, 
 the infinitesimal character of the trivial one-dimensional representation 
 of $G=O(p,q)$ is given by
\[
  (\frac{p+q-2}2, \frac{p+q-4}2, \cdots, \kappa(p+q)).  
\]

Hence we may also reformulate the branching laws in Proposition \ref{prop:disc}
 as follows.  

\medskip
\begin{prop}
[Version 2]
\label{prop:disc2}
Suppose $\lambda$ is a regular parameter
 (Definition \ref{def:admissible}).  
Then an irreducible representation $\pi$
 of $H=O(p-1,q)$
 in the discrete spectrum
 of the restriction of $\Pi_{-,\lambda}^{p,q}$
 must be isomorphic to $\pi_{-,\mu}$ for some regular parameter $\mu$, 
 and the infinitesimal characters have the interlacing property
\begin{equation}
\label{eqn:inter}
   \mu > \lambda > \frac{p+q-4}{2}> \cdots > \frac 1 2 >0.  
\end{equation}
Conversely,
 $\pi=\pi_{-,\mu}$ occurs in the discrete spectrum 
 of the restriction $\Pi_{-,\lambda}^{p,q}|_H$ 
 if the interlacing property \eqref{eqn:inter} is satisfied.  
\end{prop}

\medskip

\begin{convention}
\label{conv:deco}
We say
 that the restriction of the representation $\Pi_{-,\lambda}$ of $G$
 to $H=O(p-1,q)$
 is of discretely decomposable type.  
\end{convention}

\subsection{Branching laws for the restriction
 of $\Pi_{+,\lambda}$ to \\ $H=O(p-1,q)$ --- finite type}
\label{subsec:IV.3}
This section treats the restriction $\Pi_{+,\lambda}|_H$
 which is {\it{not}} discretely decomposable.  
We use \cite[Example 1.2 (2)]{K} which determines the whole discrete spectrum
 in the restriction $\Pi_{+,\lambda}|_H$.  
A large part of discrete summands are also obtained in \cite{OS}
 using different techniques.

The restriction $\Pi_{+,\lambda}|_H$ contains
 at most finitely many irreducible summands.  
We recall from \cite[Thm.~1.1]{K}
 (or \cite[Ex.~1.2 (2)]{K}), 
 an irreducible representation $\pi$ of $H \times O(1,0) =O(p-1,q) \times O(1)$  occurs
 in the discrete spectrum of the restriction of $\Pi_{+,\lambda}$
 if and only  it is of the form
\[
\pi ^{p-1,q}_ {+,\lambda -n-\frac 1 2} \boxtimes (\operatorname{sgn})^n \ \mbox{ for some } 0 \leq n <\lambda - \frac{1}{2}, 
\]
where $\operatorname{sgn}$ stands
 for the nontrivial character of $O(1)$.  

\medskip
\begin{prop}
[Version 1]
\label{prop:finite}
An irreducible representation $\pi$ of $H=O(p-1,q)$ occurs
 in the discrete spectrum of the restriction
 of $\Pi_{+,\lambda}$ of $G=O(p,q)$
 when restricted to $H$ if  and only if it is of the form
\[
   \pi_{+, \lambda-\frac 12 -n}^{p-1,q}
   \mbox{ where }  \lambda  -\frac 12  -n \mbox{ for } 0 \leq n <\lambda - \frac{1}{2}.
\]
\end{prop}

\medskip
\begin{remark}
\label{rem:ponecont}
There does not exist discrete spectrum
 in the restriction $\Pi_{+, \lambda}|_H$ if $p=2$.  
In fact $\pi_{+, \mu}^{1,q}$ is zero for all $\mu$ if $q \ge 1$, 
 see Remark \ref{rem:pone}.  
\end{remark}

\medskip
\begin{remark}
\label{rem:finite}
The representation $\pi_{+, \lambda-\frac 12 -n}^{p-1,q}$
 has a regular parameter, 
 or equivalently, 
 has the same infinitesimal  character as a finite-dimensional representation
 iff  
\[ 
   \lambda-\frac 12 -n > \frac{p+q-5}{2}.  
\] 
\end{remark}

\medskip
\begin{remark}
\label{rem:finadm}
In contrast to the discretely decomposable case
 (Remark \ref{rem:disc}), 
 Proposition \ref{prop:finite} tells that the implication
\[
  \text{$\lambda$ regular} \Rightarrow \text{$\mu$ regular}
\]
does not necessarily hold 
 when $\operatorname{Hom}_H(\pi_{+, \mu}, \Pi_{+, \lambda}|_H) \ne \{0\}$, 
 see Remark \ref{rem:finite} above.  
\end{remark}

\medskip

We observe that  for these representations  the  condition in the proposition  depends only on $p+q$
 and  thus the proposition for these representations does not depends
 on the inner form $SO(r,s)$ of $SO(p+q,\bC)$
 when $r+s=p+q$ with $r \ge 3$.

Recall that the infinitesimal character
 of the representation $\Pi_{+,\lambda}$ is 
\begin{equation}
\label{eqn:Ginf2}
   (\lambda, \frac{p+q -4}{2}, \dots, \kappa(p+q))  
\end{equation}
and the infinitesimal character of the representations in 
$ \pi_{+,\mu} $
\begin{equation}
\label{eqn:Hinf2}
(\mu, \frac{p+q-5}{2}, \dots, \kappa(p+q-1))
\end{equation}
as in \eqref{eqn:Ginf} and \eqref{eqn:Hinf}.

\medskip
\begin{prop}
[Version 2]
\label{prop:finite2}
Suppose $\pi$ is an irreducible unitary representation
 of $H=O(p-1,q)$.  
If $\pi$ occurs in the discrete spectrum
 of the restriction of $\Pi_{+,\lambda}$ to $H$, 
 then $\pi$ must be isomorphic to $\pi_{+,\mu}$ 
 for some $\mu>0$ with $\mu \in {\mathbb{Z}}+\frac 1 2(p+q-1)$.  
Assume further 
that $\lambda$ and $\mu$ are regular.  
Then $\pi_{+,\mu}$ occurs in the discrete spectrum
 of the restriction $\Pi_{+,\lambda}|_H$ 
 if and only if the two infinitesimal characters \eqref{eqn:Ginf2} and \eqref{eqn:Hinf2} have the interlacing property
\begin{equation}
\label{eqn:inter2}
   \lambda > \mu > \frac{p+q-4}{2}> \cdots>\frac 1 2>0.
\end{equation}
\end{prop}

\begin{remark} Consider the example:
 $q=0$ and so $G$ is compact. 
The representation $\Pi_{-,\lambda}$ is finite-dimensional 
and  has highest weight 
\[ (
\lambda -\frac{p}{2}, 0, \dots, 0)  \] for an integer $\lambda $.
A representation $\pi_{-,\mu}$ is a summand of the restriction
 to $H=SO(p-1)$
 if it has highest weight 
\[ 
   (\mu-\frac{p-1}{2}, 0,\dots, 0)
\]
for $\mu \in \bN +\frac{1}{2}$ with $\mu \geq \frac{p-1}{2} $ and 
 $\lambda-\frac{p}{2} \geq \mu -\frac{p-1}{2}  \geq 0$, 
{\it{i.e.,}} if there exists and integer $n \in \bN$ so that $
\mu = \lambda -\frac{1}{2} -n \geq \frac 1 2 (p-1)$.  
\end{remark}

This motivates the following:

\begin{convention}
\label{conv:finite}
We say that the restriction of the representation
 $\Pi_{-,\lambda}$ to $H=SO(p-1,q)$
 is of finite type.  
\end{convention}

\section{\bf The main theorems}

We retain Assumption ${\mathcal{O}}$, 
 namely, 
 $p \ge 3$ and $q \ge 2$.  
Combing the branching laws in the previous section 
 proves the conjectures in \cite[Sect.~V]{OS}
 and suggests a generalization of a conjecture
 by B.~Gross and D.~Prasad \cite{GP}, 
 which was formulated for tempered representations. 

\subsection{Results  for pairs $(O(p,q),O(p-1,q))$}

\begin{theorem} 
[Version 1]
Suppose that $\lambda$ and $\mu$ are regular parameters
(Definition \ref{def:admissible}).

\begin{enumerate}
\item
Let $\Pi_{\lambda}$ be a representations
 in the packet $\{ \Pi_{+,\lambda}, \Pi_{-,\lambda} \}$. 
There exists exactly one representations  $\pi_{\mu }$
 in the packet $\{\pi_{+,\mu}, \pi_{-,\mu} \}$
 so that
\[
   \dim \operatorname{Hom}_{H}(\Pi_{\lambda}|_H,\pi_{\mu}) = 1.
\]

\item 
Let $\pi_{\mu }$ be in the packet $\{\pi_{+,\mu}, \pi_{-,\mu} \}$. 
There exists exactly one representation  $\Pi_{\lambda} $ in the packet $\{ \Pi_{+,\lambda}, \Pi_{-,\lambda} \}$ so that 
\[
   \dim \operatorname{Hom}_{H}(\Pi_{\lambda}|_H,\pi_{\mu}) = 1.
\]
\end{enumerate}
\end{theorem} 

\medskip
Equivalently we may formulate the results
 in terms of reducible representations $U(\lambda)$
 and $V(\mu)$
 defined in \eqref{eqn:U} and \eqref{eqn:V} as follows:
 
 \medskip
 
\begin{theorem} 
[Version 2]
Suppose that $\lambda $ and  $\mu $ are regular parameters. Then
\[\dim \operatorname{Hom}_H( U(\lambda)|_H, V(\mu))= 1. \]
\end{theorem}

\medskip
We may formulate the results in interlacing properties of parameter the infinitesimal characters similar to the results in \cite{GP}.

Recall that the infinitesimal character of the representations
 of $G$
 in the packet 
 $\{ \Pi_{+,\lambda}, \Pi_{-,\lambda} \}$ is 
\[ 
   (\lambda, \frac{p+q -4}{2}, \dots ,\kappa(p+q))  
\] 
and the infinitesimal character of the representations of the subgroup $H$
 in the packet
$\{\pi_{+,\mu}, \pi_{-,\mu} \}$ is
\[ 
(\mu, \frac{p+q-5}{2}, \dots ,\kappa(p+q-1)), 
\] 
where we recall
 $(\kappa(p+q), \kappa(p+q-1))=(0,\frac 1 2)$
 if $p+q$ is even, 
 $=(\frac 1 2, 0)$ if $p+q$ is odd.  
 
 \medskip

\begin{theorem}
[Version 3]
Suppose that $\lambda $ and  $\mu $ are regular parameters. 
\begin{enumerate}
\item If the two infinitesimal characters satisfy the following interlacing property:
\[ \mu > \lambda >  \frac{p+q -4}{2}>  \dots >\frac{1}{2} >0  \] 
 then
\[\dim \operatorname{Hom}_{H}(\Pi_{-,\lambda}|_H,\pi_{-,\mu}) = 1.  \]

\item
If the two infinitesimal characters satisfy the following interlacing property:
\[ 
   \lambda> \mu >  \frac{p+q -4}{2}>  \dots  >\frac{1}{2} >0  
\] 
 then
\[
   \dim \operatorname{Hom}_{H}(\Pi_{+,\lambda}|_H,\pi_{+,\mu}) = 1.
\]
\end{enumerate}
\end{theorem}

\medskip

\begin{remark}
The trivial representation $\bf 1$ of $H =O(p-1,q)$ is in the dual of the smooth representation 
$\Pi_{+,\lambda}^{\infty}$ but not in the dual of $\Pi_{-,\lambda}^{\infty}$. There is  no other representation in the \lq\lq{packet}\rq\rq\ of the trivial representation  of $H$ and so we deduce
\[
   \mbox{dim Hom}_{H}(U(\lambda)^\infty|_H, {\bf{1}})=1, 
\]
or equivalently there is exactly one representation $\Pi_{\lambda} $ in the set $\{\Pi_{+,\lambda}^{\infty}, \Pi_{-,\lambda}^{\infty}  \}$ so that 
\[
   \mbox{dim Hom}_{H}(\Pi_{\lambda}^\infty|_H, {\bf{1}})=1.  
\]
\end{remark}

\end{document}